\documentclass[12pt, leqno]{amsart}
\usepackage{amsmath, amssymb, latexsym}
\usepackage{mathrsfs}
\usepackage{enumerate}
\usepackage{color}
\usepackage[colorlinks=true, pdfstartview=FitV, linkcolor=blue,%
citecolor=blue, urlcolor=blue]{hyperref}
\usepackage[all, knot]{xy}
\xyoption{arc}

\numberwithin{equation}{section}

\newtheorem{theorem}{Theorem}[section]
\newtheorem{corollary}[theorem]{Corollary}
\newtheorem{lemma}[theorem]{Lemma}
\newtheorem{proposition}[theorem]{Proposition}

\theoremstyle{definition}
\newtheorem{definition}[theorem]{Definition}
\newtheorem{remark}[theorem]{Remark}

\newcommand{\Hom}{\operatorname{Hom}}
\newcommand{\Ker}{\operatorname{Ker}}

\newcommand{\A}{\mathbf{A}}

\newcommand{\C}{\mathbf{C}}

\newcommand{\Q}{\mathbf{Q}}

\newcommand{\Z}{\mathbf{Z}}

\newcommand{\g}{\mathfrak{g}}

\title[Higher Representation Theory and Schur-Weyl Duality]
{Higher Representation Theory and \\
Quantum Affine Schur-Weyl Duality}

\author[Seok-Jin Kang]
{Seok-Jin Kang}

\address{Department of Mathematical Sciences
         and
         Research Institute of Mathematics \\
         Seoul National University \\ 599 Gwanak-Ro \\Seoul 151-747, Korea}

         \email{sjkang@snu.ac.kr}

\thanks{This work was supported by NRF Grant
\#2013-035155 and NRF Grant \#2013-055408}



\begin{document}

\begin{abstract}

In this article, we explain the main philosophy of 2-representation
theory and quantum affine Schur-Weyl duality. The
Khovanov-Lauda-Rouquier algebras play a central role in both themes.

\end{abstract}

\maketitle

\section*{Introduction}

The {\it Khovanov-Lauda-Rouquier algebras}, introduced by
Khovanov-Lauda \cite{KL09, KL11} and Rouquier \cite{R08, R11}, are a
family of $\Z$-graded algebras that provide a fundamental framework
for {\it 2-representation theory} and {\it quantum affine Schur-Weyl
duality}.

Let $H_{k}(\zeta)$ be the finite Hecke algebra with $\zeta$ a
primitive $n$-th root of unity and let $U_{q}(A_{n-1}^{(1)})$ be the
quantum affine algebra of type $A_{n-1}^{(1)}$. In \cite{LLT},
Lascoux-Leclerc-Thibon discovered a recursive algorithm of computing
Kashiwara's lower global basis(=Lusztig's canonical basis)
(\cite{Kas91, Lus90}) and conjectured that the coefficient
polynomials, when evaluated at $q=1$, give the composition
multiplicities of simple $H_{k}(\zeta)$-modules inside Specht
modules.

In \cite{Ar}, Ariki came up with a proof of the
Lascoux-Leclerc-Thibon conjecture using the idea of {\it
categorification}. More precisely, let $\Lambda$ be a dominant
integral weight associated with the affine Cartan datum of type
$A_{n-1}^{(1)}$ and let $H_{k}^{\Lambda}(\zeta)$ be the
corresponding cyclotomic Hecke algebra. Let
$\text{proj}\,(H_{k}^{\Lambda}(\zeta))$ denote the category of
finitely generated projective $H_{k}^{\Lambda}(\zeta)$-modules and
let $K(\text{proj}\,(H_{k}^{\Lambda}(\zeta)))$ be the Grothendieck
group of $\text{proj}\,(H_{k}^{\Lambda}(\zeta))$. Then Ariki proved
$$\bigoplus_{k=0}^{\infty}
K(\text{proj}\,(H_{k}^{\Lambda}(\zeta)))_{\C} \cong V(\Lambda),$$
where $V(\Lambda)$ is the integrable highest weight module over
$A_{n-1}^{(1)}$. Moreover, he showed that the isomorphism classes of
projective indecomposable modules correspond to the lower global
basis of $V(\Lambda)$ at $q=1$, from which the
Lascoux-Leclerc-Thibon conjecture follows.

The idea of categorification, which was originated from \cite{CF94},
can be explained as follows. In the classical representation theory,
we study the properties of an algebra $A$ that are reflected on
various vector spaces $V$. That is, we investigate various algebra
homomorphisms $\phi: A \rightarrow \text{End}(V)$. We identify $A$
with a category having a single object and its elements as
morphisms. Similarly, we consider $\text{End}(V)$ as a category with
$V$ as its object and linear operators on $V$ as morphisms. Then the
classical representation theory can be understood as the study of
functors from a category to another, whence the {\it
1-representation theory}.

We now {\it categorify} the classical representation theory. Let
$A=\bigoplus_{\alpha \in Q} A_{\alpha}$ be a graded algebra and let
$V=\bigoplus_{\lambda \in P} V_{\lambda}$ be a graded $A$-module,
where $Q$ and $P$ are appropriate abelian groups. We construct
2-categories $\mathfrak{A}$ and $\mathfrak{B}$ whose objects are
certain categories $\mathcal{A}_{\alpha}$ $(\alpha \in Q)$ and
$\mathcal{B}_{\lambda}$ $(\lambda \in P)$ such that
$$\bigoplus_{\alpha \in Q} K(\mathcal{A}_{\alpha}) \cong A, \qquad
\bigoplus_{\lambda \in P} K(\mathcal{B}_{\lambda}) \cong V.$$
We now
investigate the properties of 2-functors $R :\mathfrak{A}
\rightarrow \mathfrak{B}$. That is, by categorifying the classical
representation theory, we obtain the {\it 2-representation theory},
the study of 2-functors from a 2-category to another.

So far, one of the most interesting developments in 2-representation
theory is the one via Khovanov-Lauda-Rouquier algebras. The
Khovanov-Lauda-Rouquier algebras categorify the negative half of
quantum groups associated with {\it all} symmetrizable Cartan datum
\cite{KL09, KL11, R08, R11}. Moreover, the cyclotomic
Khovanov-Lauda-Rouquier algebras give a categorification of {\it
all} integrable highest weight modules \cite{KKas12}. Hence
Khovanov-Lauda-Rouquier's and Kang-Kashiwara's categorification
theorems provide a vast generalization of Ariki's categorification
theorem. (See also \cite{Web}.)

When the Cartan datum is symmetric, as was conjectured by
Khovanov-Lauda \cite{KL11}, Varagnolo-Vasserot proved that the
isomorphism classes of simple modules (respectively, projective
indecomposable modules) correspond to upper global basis(=dual
canonical basis) (respectively, lower global basis) \cite{VV09}.
However, when the Cartan datum is not symmetric, the above
statements do not hold in general. It is a very interesting problem
to characterize the {\it perfect basis} and {\it dual perfect basis}
that correspond to simple modules and projective indecomposable
modules, respectively.

On the other hand, the Khovanov-Lauda-Rouquier algebras can be
viewed as a huge generalization of affine Hecke algebras in the
context of {\it Schur-Weyl duality}. The Schur-Weyl duality,
established by Schur and others (see, for example, \cite{S1, S2}),
reveals a deep connection between the representation theories of
symmetric groups and general linear Lie algebras. Let $V=\C^{n}$ be
the vector representation of the general linear Lie algebra $gl_{n}$
and consider the $k$-fold tensor product of $V$. Then $gl_{n}$ acts
on $V^{\otimes k}$ by comultiplication and the symmetric group
$\Sigma_{k}$ acts on $V^{\otimes k}$ (from the right) by place
permutation. Clearly, these actions commute with each other. The
Schur-Weyl duality states that there exists a surjective algebra
homomorphism
$$\phi_{k}: \C \Sigma_{k} \longrightarrow \text{End}_{gl_{n}} (V^{\otimes k}),$$
where $\text{End}_{gl_{n}} (V^{\otimes k})$ denotes the centralizer
algebra of $V^{\otimes k}$ under the $gl_{n}$-action. Moreover,
$\phi_{k}$ is an isomorphism whenever $k \le n$.

The Schur-Weyl duality can be rephrased as follows. There is a
functor $\mathcal {F}$ from the category of finite dimensional
$\Sigma_{k}$-modules to the category of finite dimensional
polynomial representations of $gl_{n}$ given by$$M \longmapsto
V^{\otimes k} \otimes_{\C \Sigma_{k}} M,$$ where $M$ is a finite
dimensional $\Sigma_{k}$-module. The functor $\mathcal{F}$ is called
the {\it Schur-Weyl duality functor} and it defines an equivalence
of categories whenever $k \le n$.

In \cite{Jimbo86}, Jimbo extended the Schur-Weyl duality to the
quantum setting: $\Sigma_{k}$ is replaced by the finite Hecke
algebra $H_{k}$ and $gl_n$ is replaced by the quantum group
$U_{q}(gl_{n})$. Then he obtained the {\it quantum Schur-Weyl
duality functor} from the category of finite dimensional
$H_{k}$-modules to the category of finite dimensional polynomial
representations of $U_{q}(gl_{n})$, which also defines an
equivalence of categories whenever $k \le n$.

In \cite{CP96, Che, GRV94}, Chari-Pressley, Cherednik and
Ginzburg-Reshetikhin-Vasserot constructed a {\it quantum affine
Schur-Weyl duality functor} which relates the category of finite
dimensional representations of affine Hecke algebra
$H_{k}^{\text{aff}}$ and the category of finite dimensional
integrable $U_{q}'(A_{n-1}^{(1)})$-modules. The main ingredients of
their constructions are (i) the fundamental representation
$V(\varpi_{1})$, (ii) the $R$-matrices on the tensor products of
$V(\varpi_{1})$ satisfying the Yang-Baxter equation, (iii) the
intertwiners in $H_{k}^{\text{aff}}$ satisfying the braid relations.

Using Khovanov-Lauda-Rouquier algebras, one can construct  quantum
affine Schur-Weyl duality functors in much more generality. In
\cite{KKasKim1}, Kang, Kashiwara and Kim constructed such a functor
which relates the category of finite dimensional modules over
symmetric Khovanov-Lauda-Rouquuier algebras and the category of
finite dimensional integrable modules over {\it all} quantum affine
algebras. Roughly speaking, the basic idea can be explained as
follows. Using a family of {\it good} modules and $R$-matrices, we
determine a quiver $\Gamma$ and construct a symmetric
Khovanov-Lauda-Rouquier algebra $R^{\Gamma}(\beta)$ $(\beta \in
Q_{+})$. We then construct a $(U_{q}'(\g),
R^{\Gamma}(\beta))$-bimodule $\widehat{V}^{\otimes \beta}$, a
completed tensor power arising from good modules, and define the
quantum affine Schur-Weyl duality functor ${\mathcal F}$  by
$$M \mapsto \widehat{V}^{\otimes \beta} \otimes_{R^{\Gamma}(\beta)} M,$$
where $M$ is an $R^{\Gamma}(\beta)$-module.

Various choices of quantum affine algebras and good modules would
give rise to various quantum affine Schur-Weyl duality functors. We
believe that our general approach will generate a great deal of
exciting developments in the forthcoming years.

\vskip 3mm

{\it Acknowledgements}: The author is very grateful to Masaki
Kashiwara, Myungho Kim and Se-jin Oh for many valuable discussions
and suggestions on this paper.

\vskip 5mm

\section{Quantum groups}

We begin with a brief recollection of representation theory of
quantum groups.

Let $I$ be a finite index set. An integral matrix $A=(a_{ij})_{i,j
\in I}$ is called a {\it symmetrizable Cartan matrix} \, if (i)
$a_{ii}=2$ for all $i \in I$, \, (ii) $a_{ij} \le 0$ for $i \neq j$,
\, (iii) $a_{ij}=0$ if and only if $a_{ji}=0$, \, (iv) there exists
a diagonal matrix $D =\text{diag} (d_{i} \in \Z_{>0} \mid i \in I)$
such that $DA$ is symmetric.

A {\it Cartan datum} consists of :

\  (1) a symmetrizable Cartan matrix $A=(a_{ij})_{i,j \in I}$,

\  (2) a free abelian group $P$ of finite rank, the {\it weight
latice},

\  (3) $\Pi=\{\alpha_{i} \in P  \mid i \in I \}$, the set of {\it
simple roots},

\  (4) $P^{\vee} := \Hom(P, \Z)$, the {\it dual weight lattice},

\  (5) $\Pi^{\vee}=\{h_i \in P^{\vee} \mid i \in I \}$, the set of
{\it simple coroots}

\noindent satisfying the following properties

\hskip 8mm (i) $\langle h_i, \alpha_j \rangle = a_{ij}$ for all $i,
j \in I$,

\hskip 8mm (ii) $\Pi$ is linearly independent,

\hskip 8mm (iii) for each $i \in I$, there exists an element
$\Lambda_{j} \in P$ such that $$\langle h_i, \Lambda_j \rangle =
\delta_{ij} \ \ \text{for all} \ i, j \in I.$$ The $\Lambda_i$'s $(i
\in I)$ are called the {\it fundamental weights}.

We denote by
$$P^{+}:=\{\Lambda \in P \mid \langle h_i, \Lambda \rangle \ge 0 \
\text{for all} \ i \in I \}$$ the set of {\it dominant integral
weights}. The free abelian group $Q:= \bigoplus_{i \in I} \Z
\alpha_i$ is called the {\it root lattice}. Set $Q_{+} = \sum_{i\in
I} \Z_{\ge 0} \alpha_i$. For $\beta = \sum k_i \alpha_i \in Q_{+}$,
we define its {\it height} to be $|\beta|:=\sum k_i$. .

Since $A$ is symmetrizable, there exists a symmetric bilinear form
$( \ , \ )$ on ${\mathfrak h}^{*}:=\Q \otimes_{\Z} P^{\vee}$
satisfying
$$(\alpha_i, \alpha_j) = d_i a_{ij}, \quad \langle h_i, \lambda
\rangle = \dfrac{2 (\alpha_i, \lambda)}{(\alpha_i, \alpha_i)} \quad
\text{for all} \ \  \lambda \in {\mathfrak h}^{*}, \ i, j \in I.$$

Let $q$ be an indeterminate and set $q_i = q^{d_i}$ $(i \in I)$. For
$m, n \in \Z_{\ge 0}$, we define
$$[n]_{i}:= \dfrac{q_{i}^{n} - q_{i}^{-n}}{q_{i} - q_{i}^{-1}},
\qquad  [n]_{i}! := \prod_{k=1}^{n} [k]_{i}. $$ We write
$e_{i}^{(k)}: = e_{i}^k \big/[k]_{i}!$, \ $f_{i}^{(k)}: = f_{i}^k
\big /[k]_{i}!$ \ $(k \in \Z_{\ge 0}, \, i \in I)$ for the {\it
divided powers}.

\begin{definition}
The {\it quantum group} $U_{q}(\g)$ corresponding to a Cartan datum
$(A,P, \Pi,P^{\vee}, \Pi^{\vee})$ is the associative algebra over
$\Q(q)$ generated by the elements $e_i$, $f_i$ $(i \in I)$, \,
$q^{h}$ $(h \in P^{\vee})$ with defining relations
\begin{equation}\label{eq:q-group}
\begin{aligned}
& q^{0}=1, \quad q^{h} q^{h'} = q^{h+h'} \ \ (h, h' \in P^{\vee}), \\
& q^{h} e_i q^{-h} = q^{\langle h, \alpha_i \rangle} e_i, \quad
q^{h} f_i q^{-h} = q^{-\langle h, \alpha_i \rangle} f_i \ \ (h \in
P^{\vee}, \, i\in I), \\
& e_i f_j - f_j e_i = \delta_{ij} \dfrac{K_{i} - K_{i}^{-1}}{q_i -
q_i^{-1}} \ \ (K_i = q^{d_i h_i}, \, i \in I), \\
& \sum_{k=0}^{1-a_{ij}} (-1)^k e_{i}^{(1-a_{ij}-k)} e_{j}
e_{i}^{(k)} = 0 \quad (i \neq j), \\
& \sum_{k=0}^{1-a_{ij}} (-1)^k f_{i}^{(1-a_{ij}-k)} f_{j}
f_{i}^{(k)} = 0 \quad (i \neq j).
\end{aligned}
\end{equation}
\end{definition}
\noindent

Let $U_q^{0}(\g)$ be the subalgebra of $U_q(\g)$ generated by $q^h$
$(h \in P^{\vee})$ and let $U_q^{+}(\g)$ (respectively,
$U_q^{-}(\g)$) be the subalgebra of $U_q(\g)$ generated by $e_i$
(respectively, $f_i$) for all $i \in I$. Then the algebra $U_q(\g)$
has the {\it triangular decomposition}
$$U_q(\g) \cong U_q^{-}(\g) \otimes U_q^{0}(\g) \otimes
U_q^{+}(\g).$$

Let $\A = \Z[q, q^{-1}]$. We define the {\it integral form}
$U_{\A}(\g)$ of $U_q(\g)$ to be the $\A$-subalgebra of $U_q(\g)$
generated by $e_i^{(k)}$, $f_i^{(k)}$, $q^h$ $(i \in I, \, h \in
P^{\vee},\, k \in \Z_{\ge 0})$. Let $U_{\A}^{0}(\g)$ be the
$\A$-subalgebra of $U_q(\g)$ generated by $q^h$ $(h \in P^{\vee})$
and let $U_{\A}^{+}(\g)$ (respectively, $U_{\A}^{-}(\g)$) be the
$\A$-subalgebra of $U_{q}(\g)$ generated by $e_i^{(k)}$
(respectively, $f_i^{(k)}$) $(i \in I, \, k \in \Z_{\ge 0})$. Then
we have
$$U_{\A}(\g) \cong U_{\A}^{-}(\g) \otimes U_{\A}^{0}(\g)
\otimes U_{\A}^{+}(\g).$$

A $U_q(\g)$-module $V$ is called a {\it highest weight module with
highest weight $\Lambda \in P$} if there exists a nonzero vector
$v_{\Lambda}$ in $V$, called the {\it highest weight vector}, such
that

\ \ \ \ (i) $e_{i}\, v_{\Lambda} =0$ for all $i \in I$,

\ \ \ \ (ii) $q^{h}\, v_{\Lambda} = q^{\langle h, \Lambda \rangle}
v_{\Lambda}$ for all $h \in P^{\vee}$,

\ \ \ \ (iii) $V = U_q(\g)\, v_{\Lambda}$.

For each $\Lambda \in P$, there exists a unique irreducible highest
weight module $V(\Lambda)$ with highest weight $\Lambda$. The {\it
integral form} of $V(\Lambda)$ is defined to be
$$V_{\A}(\Lambda):= U_{\A}(\g) \, v_{\Lambda},$$
where $v_{\Lambda}$ is the highest weight vector.

Consider the anti-involution $\phi:U_q(\g) \rightarrow U_q(\g)$
defined by
$$q^h \mapsto q^h, \qquad e_i \mapsto f_i, \qquad f_i \mapsto e_i
\quad (h \in P^{\vee}, \ i \in I).$$ Then there exists a unique
non-degenerate symmetric bilinear form $( \ , \ )$ on $V(\Lambda)$
satisfying
\begin{equation} \label{eq:bilinear form}
(v_{\Lambda}, v_{\Lambda})=1, \quad (x \,u , v) = (u, \phi(x) \, v)
\quad \text{for all} \ x \in U_q(\g), \, u, v \in V(\Lambda).
\end{equation}
The {\it dual} of $V_{\A}(\Lambda)$ is defined to be
$$V_{\A}(\Lambda)^{\vee}:= \{v \in V(\Lambda) \mid \, (u, v) \in \A
\ \ \text{for all} \ u \in V_{\A}(\Lambda) \}.$$ Note that
$V_{\A}(\Lambda)^{\vee}_{\lambda} =
\Hom_{\A}(V_{\A}(\Lambda)_{\lambda}, \A)$ for all $\lambda \in P$.

The {\it Category} ${\mathcal O}_{\text{int}}$ consists of
$U_q(\g)$-modules $M$ such that

(i) $M=\bigoplus_{\mu \in P}M_{\mu}$, where $M_{\mu}:=\{m \in M \mid
\, q^h \, m = q^{\langle h, \mu \rangle} \, m \ \text{for all} \ h
\in P^{\vee} \}, $

(ii) $e_i, f_i\, (i\in I)$ are locally nilpotent on $M$,

(iii) there exist finitely many elements $\lambda_{1}, \ldots,
\lambda_{s} \in P$ such that $$\text{wt}(M):= \{ \mu \in P \mid \,
M_{\mu} \neq 0 \} \subset \bigcup_{j=1}^{s} (\lambda_{j} - Q_{+}).$$

The following properties of the category ${\mathcal O}_{\text{int}}$
are well-known. (See, for example, \cite{HK02, Kac, Lus93}.)

\begin{proposition} \hfill

{\rm

(a) The category ${\mathcal O}_{\text{int}}$ is semisimple.

(b) The $U_q(\g)$-module $V(\Lambda)$ with $\Lambda \in P^{+}$
belongs to  ${\mathcal O}_{\text{int}}$.

(c) Every simple object in ${\mathcal O}_{\text{int}}$ has the form
$V(\Lambda)$ for some $\Lambda \in P^{+}$.
 }
\end{proposition}

\vskip 5mm

\section{Khovanov-Lauda-Rouquier algebras}\label{sec:KLR}

Let $\mathbf{k}$ be a field and let $(A, P,\Pi,
P^{\vee},\Pi^{\vee})$ be a Cartan datum.

For each $i \neq j$, set
$$S_{ij}:= \{(p,q) \in \Z_{\ge 0} \times \Z_{\ge 0} \mid \,
(\alpha_i, \alpha_i) p + (\alpha_j, \alpha_j) q = -2(\alpha_i,
\alpha_j) \}.$$ Define a family of polynomials ${\sf Q}=({\sf
Q}_{ij})_{i,j\in I}$ in $\mathbf{k}[u,v]$ by
\begin{equation}\label{eq:Q}
{\sf Q}_{ij}(u,v):= \begin{cases} 0 & \ \text{if} \ i =j, \\
\sum_{(p,q)\in S_{ij}} t_{i,j;p,q}u^p v^q & \ \text{if} \ i \neq j
\end{cases}
\end{equation}
for some $t_{i,j;p,q} \in {\mathbf k}$ such that $t_{i,j;p,q}
=t_{j,i; q,p}$ and $t_{i,j;-a_{ij},0} \in {\mathbf k}^{\times}$. In
particular,
$${\sf Q}_{ii}(u,v)=0, \quad {\sf Q}_{ij}(u,v)= {\sf Q}_{ji}(v,u) \ \ (i \neq j).$$

The symmetric group $\mathfrak{S}_{n}=\langle s_1, s_2, \ldots,
s_{n-1}\rangle$ acts on $I^{n}$ by place permutation, where $s_i$
denotes the transposition $(i, i+1)$.

\begin{definition}\label{def:KLR}
The {\it Khovanov-Lauda-Rouquier algebra} $R(n)$ of degree $n\ge 0$
associated with $(A, \sf{Q})$ is the associative algebra over
$\mathbf{k}$ generated by the elements $e(\nu)$ $(\nu \in I^n)$,
$x_{k}$ $(1 \le k \le n)$, $\tau_{l}$ $(1 \le l \le n-1)$ with
defining relations
\begin{equation}\label{eq:KLR}
\begin{aligned}
& e(\nu) e(\nu') = \delta_{\nu, \nu'} e(\nu), \quad
\sum_{\nu \in  I^n } e(\nu) = 1, \\
& x_{k} x_{l} = x_{l} x_{k}, \quad  x_{k} e(\nu) = e(\nu) x_{k}, \\
& \tau_{l} e(\nu) = e(s_{l}(\nu)) \tau_{l}, \quad \tau_{k} \tau_{l}
=\tau_{l} \tau_{k} \ \ \text{if} \ |k-l|>1, \\
& \tau_{k}^2 e(\nu) = Q_{\nu_{k}, \nu_{k+1}} (x_{k}, x_{k+1})
e(\nu), \\
& (\tau_{k} x_{l} - x_{s_k(l)} \tau_{k}) e(\nu) = \begin{cases}
-e(\nu) \ \ & \text{if} \ l=k, \,\nu_{k} = \nu_{k+1}, \\
e(\nu) \ \ & \text{if} \ l=k+1,\, \nu_{k}=\nu_{k+1}, \\
0 \ \ & \text{otherwise},
\end{cases} \\
& (\tau_{k+1} \tau_{k} \tau_{k+1}-\tau_{k} \tau_{k+1} \tau_{k}) e(\nu)\\
& \quad=\begin{cases} \dfrac{Q_{\nu_{k}, \nu_{k+1}}(x_{k}, x_{k+1})
- Q_{\nu_{k}, \nu_{k+1}}(x_{k+2}, x_{k+1})} {x_{k} - x_{k+2}}e(\nu)
\ \ & \text{if} \
\nu_{k} = \nu_{k+2}, \\
0 \ \ & \text{otherwise}.
\end{cases}
\end{aligned}
\end{equation}
\end{definition}
\noindent
The algebra $R(n)$ has a $\Z$-grading by assigning the
degrees as follows:
$$\text{deg}\, e(\nu)=0, \qquad \text{deg}\, x_{k} e(\nu) =
(\alpha_{\nu_k}, \alpha_{\nu_{k}}), \qquad \text{deg}\, \tau_{l}
e(\nu) = -(\alpha_{\nu_{l}}, \alpha_{\nu_{l+1}}).$$

We denote by $q$ the {\it degree-shift functor} defined by $$(qM)_k
= M_{k-1},$$ where $M=\bigoplus_{k \in \Z} M_{k}$ is a graded
$R(n)$-module. Also there is an algebra involution $\psi:R(n)
\rightarrow R(n)$ given by
\begin{equation}\label{eq:anti}
\begin{aligned}
& e(\nu) \mapsto e(\nu'), \qquad x_{k} \mapsto x_{n-k+1}, \\
& \tau_{l}e(\nu) \mapsto \begin{cases} - \tau_{n-l}\, e(\nu') \ &
\text{if} \ \nu_{l}=\nu_{l+1}, \\
\tau_{n-l} \, e(\nu') \ & \text{if} \ \nu_{l} \neq \nu_{l+1},
\end{cases}
\end{aligned}
\end{equation}
where $\nu=(\nu_1, \nu_2, \ldots, \nu_{n})$ and $\nu'=(\nu_{n},
\ldots, \nu_{2}, \nu_{1})$.

\vskip 2mm

By the embedding $R(m) \otimes R(n) \hookrightarrow R(m+n)$, we may
consider $R(m) \otimes R(n)$ as a subalgebra of $R(m+n)$. For an
$R(m)$-module $M$ and an $R(n)$-module $N$, we define their {\it
convolution product} $M \circ N$ by
\begin{equation}\label{eq:convolution}
M \circ N := R(m+n) \otimes_{R(m) \otimes R(n)} (M \otimes N).
\end{equation}
Since $R(m+n)$ is free over $R(m) \otimes R(n)$ (\cite[Proposition
2.16]{KL09}), the bifunctor $(M,N) \mapsto M \circ N$ is exact in
$M$ and $N$.

For $n \ge 0$ and $\beta \in Q_{+}$ with $|\beta|=n$, set
\begin{equation*} \label{eq:idempotent}
I^{\beta} := \{ \nu=(\nu_1, \ldots, \nu_n) \mid \, \alpha_{\nu_1} +
\cdots + \alpha_{\nu_n} = \beta \}, \qquad e(\beta):= \sum_{\nu \in
I^{\beta}} e(\nu).
\end{equation*}
Then $e(\beta)$ is a central idempotent in $R(n)$.  We also define
\begin{equation*}\label{eq:ebeta}
e(\beta, \alpha_i):=\sum_{\substack{\nu \in I^{\beta+\alpha_i} \\
\nu_{n+1}=i }} e(\nu), \quad e(\alpha_i, \beta):=\sum_{\substack{\nu \in I^{\beta+\alpha_i} \\
\nu_{1}=i }} e(\nu).
\end{equation*}
The algebra
$$R(\beta):=R(n) e(\beta)$$ is called the {\it Khovanov-Lauda-Rouquier
algebra at $\beta$}.

\vskip 2mm

For a $\mathbf{k}$-agebra $R$, we denote by $\text{mod}(R)$
(respectively, $\text{proj}(R)$ and $\text{rep}(R)$) the category of
$R$-modules (respectively, the category of finitely generated
projective $R$-modules and the category of finite dimensional
$R$-modules).

If $R$ is a graded ${\mathbf k}$-algebra, we will use
$\text{Mod}(R)$ (respectively, $\text{Proj}(R)$ and $\text{Rep}(R)$)
for the category of graded $R$-modules (respectively, the category
of finitely generated projective graded $R$-modules and the category
of finite dimensional graded $R$-modules).

\vskip 2mm

For each $i \in I$, define the functors
\begin{equation}\label{eq:EF}
\begin{aligned}
& E_{i} : \text{Mod}\,(R(\beta+\alpha_i)) \longrightarrow \text{Mod}\,(R(\beta)), \\
& F_{i} : \text{Mod}\,(R(\beta)) \longrightarrow
\text{Mod}\,(R(\beta+ \alpha_i))
\end{aligned}
\end{equation}
by
\begin{equation}\label{eq:EFdef}
\begin{aligned}
& E_{i}(N) =  e(\beta, \alpha_i)\,R(\beta+
\alpha_i) \otimes_{R(\beta+\alpha_i)}\, N, \\
& F_{i}(M) = R(\beta+\alpha_i)\, e(\beta, \alpha_i)
\otimes_{R(\beta)} M
\end{aligned}
\end{equation}
for $M \in \text{Mod}\,(R(\beta))$, $N \in
\text{Mod}\,(R(\beta+\alpha_i))$.

By \cite[Proposition 2.16]{KL09}, the functors $E_i$ and $F_i$ are
exact and send finitely generated projective modules to finitely
generated projective modules. Hence \eqref{eq:EF} restricts to the
functors
\begin{equation}\label{eq:EFproj}
\begin{aligned}
& E_{i} : \text{Proj}\,(R(\beta+\alpha_i)) \longrightarrow \text{Proj}\,(R(\beta)), \\
& F_{i} : \text{Proj}\,(R(\beta)) \longrightarrow
\text{Proj}\,(R(\beta+ \alpha_i)).
\end{aligned}
\end{equation}

For $1 \le k < n$, set $b_{k} := \tau_{k} x_{k+1}$ and $b_{k}': =
x_{k+ 1} \tau_{k}$. Let $w_{0}= s_{i_1} \cdots s_{i_r}$ be the
longest element in $S_{n}$ and set
$${\mathbf b}(n):= b_{i_1} \cdots b_{i_r}, \quad
{\mathbf b}'(n):= b'_{i_r} \cdots b'_{i_1}.$$ For each $n \ge 0$, we
define the {\it divided powers} by
$$E_{i}^{(n)}:= {\mathbf b}'(n) E_{i}^n, \qquad F_{i}^{(n)}:=
F_{i}^n {\mathbf b}(n).$$ In \cite{KL09} and \cite{R08},
Khovanov-Lauda and Rouquier proved the following {\it
categorification theorem}.

\begin{theorem}{\rm \cite{KL09, R08}} \label{thm:categorification}

{\rm  There exists an $\A$-algebra isomorphism
$$U_{\A}^{-}(\g) \overset {\sim} \longrightarrow K(\text{Proj}(R)) \
\ \text{given by} \ \ f_{i}^{(n)} \longmapsto [F_{i}^{(n)}] \ \
(i\in I, \, n \ge 0),$$ where $K(\text{Proj}(R)):= \bigoplus_{\beta
\in Q_{+}} K(\text{Proj}(R(\beta)))$. }
\end{theorem}

Thus we have constructed a 2-category $\mathfrak {R}$ such that the
objects are the categories $\text{Proj}\,(R(\beta))$ $(\beta \in
Q_{+})$ and the categories ${\mathcal Hom}
\,(\text{Proj}\,R(\alpha), \text{Proj}\, R(\beta))$ consist of the
monomials $F_{i_1} \cdots F_{i_r}$ $(i_k \in I, \, r \ge 0)$ of
functors satisfying
$$\alpha_{i_1} + \cdots + \alpha_{i_r}=
\begin{cases}
\alpha - \beta \ & \text{if} \ \alpha \ge \beta, \\
\beta - \alpha \ & \text{if} \ \beta \ge \alpha.
\end{cases} $$
The morphisms in ${\mathcal Hom} \,(\text{Proj}\,R(\alpha),
\text{Proj}\, R(\beta))$ are the natural transformations  generated
by $x_i: F_i \rightarrow F_i$, \, $\tau_{ij}: F_i F_j \rightarrow
F_j F_i$ $(i, j \in I)$ satisfying the relations
\begin{equation*}
\begin{aligned}
& \tau_{ij} \circ \tau_{ji} = {\sf Q}_{ij}(F_j \,x_i, x_j F_i), \\
& \tau_{jk} F_i \circ F_{j}\, \tau_{ik} \circ \tau_{ij} F_k -
F_{k}\,\tau_{ij} \circ \tau_{ik} F_{j} \circ F_{i}\, \tau_{jk} \\
& \hskip 5mm = \begin{cases} \dfrac{{\sf Q}_{ij}(x_i F_j, F_j\, x_i)
F_i - F_i {\sf Q}_{ij}(F_j\, x_i, x_j F_i) }{ x_i F_j F_i - F_i
F_j\, x_i } F_i \quad & \text{if} \ \
i=k, \\
0 & \text{otherwise},
\end{cases}\\
& \tau_{ij} \circ x_i F_j - F_j \, x_i \circ \tau_{ij} =
\delta_{ij}, \\
& \tau_{ij} \circ F_i \, x_j - x_j F_i \circ \tau_{ij} =
-\delta_{ij}.
\end{aligned}
\end{equation*}
It is straightforward to verify that $\mathfrak{R}$ satisfies all
the axioms for 2-categories \cite{R08, R11}.

\vskip 2mm

For the later use, we define a functor $\overline{F_i}: \text{Mod}
\, (R(\beta)) \longrightarrow \text{Mod}\, (R(\beta + \alpha_i))$ by
$$\overline{F_i}(M):= R(\beta+\alpha_i) \,
e(\alpha_i, \beta) \otimes_{R(\beta)} M \ \ \text{for} \ i \in I, \,
M \in \text{Mod}\,(R(\beta)).$$ The properties of the functors
$E_i$, $F_i$ and $\overline{F_i}$ $(i \in I)$ are given in the
following proposition.

\begin{proposition}{\rm \cite{KKas12}} \label{prop:EF} \hfill

{\rm

(a) We have an exact sequence in $\text{Mod}\, (R(\beta))$
$$0 \longrightarrow \overline{F_i}\,E_i\, M  \longrightarrow E_i
\overline{F_i}\, M \longrightarrow q^{-(\alpha_i, \alpha_i)} M
\otimes {\mathbf k}[t_i] \longrightarrow 0$$ which is functorial in
$M \in \text{Mod}\, (R(\beta))$.

(b) There exist natural isomorphisms

$$\begin{aligned}
& E_i F_j \overset{\sim} \longrightarrow F_j E_i, \quad E_i
\overline{F_j} \overset{\sim} \longrightarrow \overline{F_j} E_i \ \
\ \text{if} \ i \neq j, \\
& E_i F_i \overset{\sim} \longrightarrow q^{-(\alpha_i, \alpha_i)}
F_i E_i \oplus {\mathbf 1} \otimes {\mathbf k}[t_i] \ \ \text{if} \
i=j,
\end{aligned}
$$
where $t_i$ is an indeterminate of degree $(\alpha_i, \alpha_i)$ and
$${\mathbf 1} \otimes {\mathbf k}[t_i] : \text{Mod}\,(R(\beta))
\rightarrow \text{Mod}\, (R(\beta))$$ is the degree-shift functor
sending $M$ to $M \otimes {\mathbf k}[t_i]$ for $M \in \text{Mod}\,
(R(\beta))$ $(\beta \in Q_{+})$. }
\end{proposition}

\vskip 5mm

\section{Cyclotomic categorification theorem}

Let $\Lambda \in P^{+}$ and let
$$a^{\Lambda}(x_1) := \sum_{\nu \in I^{\beta}} x_{1}^{\langle
h_{\nu_1}, \Lambda \rangle} e(\nu) \ \ \in R(\beta).$$ Then the {\it
cyclotomic Khovanov-Lauda-Rouquier algebra $R^{\Lambda}(\beta)$}
$(\beta \in Q_{+})$ is defined to be the quotient algebra
\begin{equation}\label{def:cycKLR}
R^{\Lambda}(\beta):= R(\beta) \big/ R(\beta) a^{\Lambda}(x_1)
R(\beta).
\end{equation}
We would like to show that the cyclotomic Khovanov-Lauda-Rouquier
algebras provide a categorification of irreducible highest weight
$U_q(\g)$-modules in the category ${\mathcal O}_{\text{int}}$.

\vskip 2mm

For each $i \in I$, define the functors
\begin{equation}\label{eq:EFla}
\begin{aligned}
& E_{i}^{\Lambda}: \text{Mod}(R^{\Lambda}(\beta+\alpha_i))
\longrightarrow \text{Mod}(R^{\Lambda}(\beta)), \\
& F_{i}^{\Lambda}: \text{Mod}(R^{\Lambda}(\beta)) \longrightarrow
\text{Mod}(R^{\Lambda}(\beta+ \alpha_i))
\end{aligned}
\end{equation}
by
\begin{equation}\label{EFladef}
\begin{aligned}
& E_{i}^{\Lambda}(N) = e(\beta, \alpha_i)R^{\Lambda}\,(\beta+
\alpha_i) \otimes_{R^{\Lambda}(\beta+\alpha_i)} N, \\
& F_{i}^{\Lambda}\,(M) = R^{\Lambda}(\beta+\alpha_i)\, e(\beta,
\alpha_i) \otimes_{R^{\Lambda}(\beta)} M
\end{aligned}
\end{equation}
for $M \in \text{Mod}(R^{\Lambda}(\beta))$, $N \in
\text{Mod}(R^{\Lambda}(\beta+\alpha_i))$. However, since
$R^{\Lambda}(\beta+\alpha_i)$ is not free over $R^{\Lambda}(\beta)$,
there is no guarantee that $E_{i}^{\Lambda}$ and $F_{i}^{\Lambda}$
send finitely generated projective modules to finitely generated
projective modules. To prove this, we need to show that
$R^{\Lambda}(\beta+\alpha_i)\, e(\beta, \alpha_i)$ is a projective
right $R^{\Lambda}(\beta)$-module.

\vskip 2mm

Let
\begin{equation*}
\begin{aligned}
F^{\Lambda}& :=R^{\Lambda}(\beta+\alpha_i)\,e(\beta, \alpha_i) \\
& \quad  = \dfrac{R(\beta + \alpha_i)\, e(\beta, \alpha_i)}
{R(\beta+\alpha_i) \, a^{\Lambda}(x_1) R(\beta + \alpha_i)\, e(\beta, \alpha_i)}, \\
K_{0} & := R(\beta+\alpha_i) \, e(\beta, \alpha_i)
\otimes_{R(\beta)} R^{\Lambda}(\beta) \\
& \quad = \dfrac{R(\beta + \alpha_i)\, e(\beta, \alpha_i)}{R(\beta +
\alpha_i)\, a^{\Lambda}(x_1) R(\beta)\, e(\beta,
\alpha_i)}, \\
K_{1} & := R(\beta+\alpha_i) \, e(\alpha_i, \beta)
\otimes_{R(\beta)} R^{\Lambda}(\beta)\\
& \quad = \dfrac{R(\beta + \alpha_i)\, e(\alpha_i, \beta)}{R(\beta +
\alpha_i)\, a^{\Lambda}(x_2) R^{1}(\beta)\, e(\alpha_i, \beta)},
\end{aligned}
\end{equation*}
where $R^{1}(\beta)$ is the subalgebra of $R(\beta+\alpha_i)$
generated by $e(\alpha_i, \nu)$ $(\nu \in I^{\beta})$, $x_{k}$
$(2\le k \le n+1)$, $\tau_{l}$ $(2 \le l \le n)$. Then
$F^{\Lambda}$, $K_{0}$ and $K_{1}$ can be regarded as
$(R(\beta+\alpha_i), R^{\Lambda}(\beta))$-bimodules.

Let $t_i$ be an indeterminate of degree $(\alpha_i, \alpha_i)$. Then
${\mathbf k}[t_i]$ acts on $R(\beta+\alpha_i)\,e(\alpha_i, \beta)$
and  $K_{1}$ from the right by $t_i = x_{1} e(\alpha_i, \beta)$. On
the other hand, ${\mathbf k}[t_i]$ acts on $K_{0}$ and $F^{\Lambda}$
from the right by $t_{i}= x_{n+1} e(\beta, \alpha_i)$. Hence all of
them have a structure of $(R(\beta+\alpha_i), R(\beta) \otimes
{\mathbf k}[t_i])$-bimodules. Moreover, $F^{\Lambda}$, $K_{0}$ and
$K_{1}$ are in fact $(R(\beta+\alpha_i), R^{\Lambda}(\beta) \otimes
{\mathbf k}[t_i])$-bimodules.

\vskip 2mm

In \cite{KL09}, it was shown that $K_{0}$ and $K_{1}$ are finitely
generated projective right $(R^{\Lambda}(\beta) \otimes {\mathbf
k}[t_i])$-modules. Let $\pi: K_{0} \rightarrow F^{\Lambda}$ be the
canonical projection and let $P:K_{1} \rightarrow K_{0}$ be the
right multiplication by $a^{\Lambda}(x_1)\,\tau_{1} \cdots
\tau_{n}$.

\vskip 2mm

The following theorem is one of the main results in \cite{KKas12}.

\begin{theorem} {\rm \cite{KKas12}} \label{thm:exact} \hfill

{\rm  The sequence
\begin{equation}\label{eq:exact}
0 \longrightarrow K_{1} \overset{P} \longrightarrow K_{0}
\overset{\pi} \longrightarrow F^{\Lambda} \longrightarrow 0
\end{equation}
is exact as $(R(\beta+\alpha_i, R^{\Lambda}(\beta) \otimes {\mathbf
k}[t_i])$-bimodules.}
\end{theorem}

\vskip 2mm

Hence we get a projective resolution of $F^{\Lambda}$ of length $1$
as a right $R^{\Lambda}(\beta)[t_i]$-module. By the following lemma,
we conclude that $F^{\Lambda}$ is a finitely generated projective
right $R^{\Lambda}(\beta)$-module.

\begin{lemma} {\rm \cite{KKas12}}
{\rm Let $R$ be a ring and let $f(t)$ be a monic polynomial in
$R[t]$ with coefficients in the center of $R$.

If an $R[t]$-module $M$ is annihilated by $f(t)$ and has projective
dimension $\le 1$, then $M$ is projective as an $R$-module. }
\end{lemma}

\vskip 3mm

Thus we obtain the following important theorem.

\begin{theorem}{\rm \cite{KKas12}} \label{thm:proj} \hfill

{\rm

(a) $R^{\Lambda}(\beta+\alpha_i)\, e(\beta, \alpha_i)$ is a
projective right $R^{\Lambda}(\beta)$-module.

(b) $e(\beta, \alpha_i) R^{\Lambda}(\beta+\alpha_i)$ is a projective
left $R^{\Lambda}$-module.

(c) The functors $E_{i}^{\Lambda}$ and $F_{i}^{\Lambda}$ are exact.

(d) The functors $E_{i}^{\Lambda}$ and $F_{i}^{\Lambda}$ send
finitely generated projective modules to finitely generated
projective modules. }
\end{theorem}

\vskip 3mm

\begin{corollary} {\rm \cite{KKas12}} \label{cor:exact} \hfill

{\rm

For all $i \in I$ and $\beta \in Q_{+}$, we have an exact sequence
of $R(\beta+\alpha_i)$-modules
$$0 \longrightarrow q^{(\alpha_i, 2\Lambda - \beta)}
\overline{F_{i}} \, M \longrightarrow F_{i}\,M \longrightarrow
F_{i}^{\Lambda}\,M \longrightarrow 0$$ which is functorial in $M \in
\text{Mod}\,R^{\Lambda}(\beta)$. }
\end{corollary}

\vskip 2mm

To complete the construction of cyclotomic categorification, it
remains to show that the adjoint pair $(F_{i}^{\Lambda},
E_{i}^{\Lambda})$ gives an {\it $sl_{2}$-categorification}
introduced by Chuang-Rouquier \cite{CR08}.

\begin{theorem}{\rm \cite{KKas12}} \label{thm:sl2} \hfill

{\rm (a) For $i \neq j$, there exists a  natural isomorphism
$$q^{-(\alpha_i, \alpha_j)} F_{j}^{\Lambda} E_{i}^{\Lambda}
\overset{\sim} \longrightarrow E_{i}^{\Lambda} F_{j}^{\Lambda}.
$$

(b) Let $\lambda = \Lambda - \beta$ $(\beta \in Q_{+})$.

\ \ \ \ (i) If $\langle h_i, \Lambda \rangle \ge 0$, there exists a
natural isomorphism
$$q_{i}^{-2} F_{i}^{\Lambda} E_{i}^{\Lambda} \oplus
\bigoplus_{k=0}^{\langle h_i, \Lambda \rangle -1} q_{i}^{2k}
{\mathbf 1} \overset{\sim} \longrightarrow
E_{i}^{\Lambda}F_{i}^{\Lambda}.$$

\ \ \ \ (ii) If $\langle h_i, \Lambda \rangle \le 0$, there exists a
natural isomorphism
$$q_{i}^{-2} F_{i}^{\Lambda} E_{i}^{\Lambda} \overset{\sim}
\longrightarrow E_{i}^{\Lambda} F_{i}^{\Lambda} \oplus
\bigoplus_{k=0}^{-\langle h_i, \lambda \rangle -1} q_i^{2k-2}
{\mathbf 1}.$$ }
\end{theorem}

\begin{proof} \ We will give a very rough sketch of the proof.
The assertion (a) can be proved in a straightforward manner.

To prove (b), note that Theorem \ref{thm:exact} and Corollary
\ref{cor:exact} yield the following commutative diagram.
\begin{equation*}
\begin{array}{c}
\xymatrix{
{} & 0 \ar[d] & 0 \ar[d]  & q_{i}^{-2}M & {} \\
0 \ar[r] & q^{(\alpha_i | 2 \Lambda - \beta)} \overline{F_{i}} E_{i}
M \ar[d] \ar[r] & q_{i}^{-2} F_{i} E_{i} M \ar[d] \ar[r]
\ar[ur]^{\varepsilon} & q_{i}^{-2} F_{i}^{\Lambda}
E_{i}^{\Lambda} M  \ar[d] \ar[r] & 0 \\
0 \ar[r] & q^{(\alpha_{i}| 2 \Lambda - \beta)} E_{i}
\overline{F_{i}} M \ar[d] \ar[r]  &
E_{i} F_{i} M \ar[d] \ar[r] & E_{i}^{\Lambda} F_{i}^{\Lambda} M  \ar[r] & 0 \\
{} & q^{(\alpha_{i} | 2 \Lambda - 2\beta)} {\mathbf k}[t_{i}]
\otimes M \ar[d] \ar[r] &
{\mathbf k}[t_{i}] \otimes M \ar[d] & {} & {} \\
{} & 0 & 0 & {} & {}
  }\end{array}
\end{equation*}

Let $A: q^{2(\alpha_i , \Lambda - \beta)}{\mathbf k}[t_i]\otimes
R^{\Lambda}(\beta) \longrightarrow {\mathbf k}[t_i]\otimes
R^{\Lambda}(\beta)$ be the $R^{\Lambda}(\beta)$-bilinear map given
by chasing the diagram. By a detailed analysis of the above
commutative diagram at the kernel level, the Snake Lemma gives the
following exact sequence of $R^{\Lambda}(\beta)$-bimodules
$$0 \longrightarrow \Ker A \longrightarrow q_{i}^{-2}
F_{i}^{\Lambda} E_{i}^{\Lambda} R^{\Lambda}(\beta) \longrightarrow
E_{i}^{\Lambda} F_{i}^{\Lambda} R^{\Lambda}(\beta) \longrightarrow
\text{Coker} A \longrightarrow 0.$$ If $\langle h_i, \lambda
\rangle\ge 0$, we have $\Ker A=0, \ \ \bigoplus_{k=0}^{a-1} {\mathbf
k}\, t_{i}^k \otimes R^{\Lambda}(\beta) \overset{\sim}
\longrightarrow \text{Coker} A,$ and if $\langle h_i, \lambda
\rangle \le 0$, then $\text{Coker} A =0, \ \ \Ker(A) =
q^{2(\alpha_i| \Lambda - \beta)} \bigoplus_{k=0}^{a-1} {\mathbf k}\,
t_{i}^k \otimes R^{\Lambda}(\beta),$ from which our assertion(b)
follows.
\end{proof}

\vskip 2mm Set
\begin{equation*}
\begin{aligned}
& K(\text{Proj}\,(R^{\Lambda})) := \bigoplus_{\beta \in Q_{+}}
K(\text{Proj}\,R^{\Lambda}(\beta)),\\
& \ K(\text{Rep}\,(R^{\Lambda})) := \bigoplus_{\beta \in Q_{+}}
K(\text{Rep}\,R^{\Lambda}(\beta)).
\end{aligned}
\end{equation*}
We define the endomorphisms ${\sf E}_{i}$ and ${\sf F}_{i}$ on
$K(\text{Proj}\,(R^{\Lambda}))$ by
\begin{equation*}
\begin{aligned}
& {\sf E}_{i}=[q_i^{1 - \langle h_i, \Lambda - \beta \rangle}
E_{i}^{\Lambda}]: K(\text{Proj}\,R^{\Lambda}(\beta + \alpha_i))
\longrightarrow K(\text{Proj}\,R^{\Lambda}(\beta)), \\
& {\sf F}_{i}=[F_{i}^{\Lambda}]: K(\text{Proj}\,R^{\Lambda}(\beta))
\longrightarrow K(\text{Proj}\,R^{\Lambda}(\beta + \alpha_i)).
\end{aligned}
\end{equation*}
On the other hand, we define ${\sf E}_{i}$ and ${\sf F}_{i}$ on
$K(\text{Rep}\,(R^{\Lambda}))$ by
\begin{equation*}
\begin{aligned}
& {\sf E}_{i}=[E_{i}^{\Lambda}]: K(\text{Rep}\,R^{\Lambda}(\beta +
\alpha_i))
\longrightarrow K(\text{Rep}\,R^{\Lambda}(\beta)), \\
& {\sf F}_{i}=[q_i^{1 - \langle h_i, \Lambda - \beta \rangle}
F_{i}^{\Lambda}]: K(\text{Rep}\,R^{\Lambda}(\beta)) \longrightarrow
K(\text{Rep}\,R^{\Lambda}(\beta + \alpha_i)).
\end{aligned}
\end{equation*}

Let ${\sf K}_{i}$ be the endomorphism on
$K(\text{Proj}\,R^{\Lambda}(\beta))$ and
$K(\text{Rep}\,R^{\Lambda}(\beta))$ given by the multiplication by
$q_{i}^{\langle h_i, \Lambda - \beta \rangle}$ for each $\beta \in
Q_{+}$. Then we have
$$[{\sf E}_{i}, {\sf F}_{j}] = \delta_{ij} \dfrac{{\sf K}_{i} - {\sf K}_{i}^{-1}}{q_i -
q_i^{-1}} \ \ \ \text{for} \, \ i, j \in I.$$

Therefore, we obtain the {\it cyclotomic categrification theorem}
for irreducible highest weight $U_q(\g)$-modules in the category
${\mathcal O}_{\text{int}}$.

\vskip 2mm

\begin{theorem} {\rm \cite{KKas12}} \hfill

\vskip 2mm

{\rm For each $\Lambda \in P^{+}$, there exist $U_{\A}(\g)$-module
isomorphisms
$$K(\text{Proj}\,R^{\Lambda}) \overset{\sim} \longrightarrow V_{\A}(\Lambda)
\ \ \text{and} \ \  K(\text{Rep}\,R^{\Lambda}) \overset{\sim}
\longrightarrow V_{\A}(\Lambda)^{\vee}.$$ }
\end{theorem}

\vskip 2mm

Therefore, for each $\Lambda \in P^{+}$, we have constructed a
2-category ${\mathfrak R}^{\Lambda}$ consisting of
$\text{Proj}\,(R^{\Lambda}(\beta))$ $(\beta \in Q_{+})$, which gives
an integrable 2-representation ${\mathfrak R}^{\Lambda}$ of
${\mathfrak R}$ in the sense of \cite{R08, R11}. (See also
\cite{Web}.)

\begin{remark}
There are several generalizations of Khovanov-Lauda-Rouquier
algebras and categorification theorems. In \cite{KKO,  KKP, KOP},
the Khovanov-Lauda-Rouquier algebras associated with
Borcherds-Cartan data have been defined and their properties have
been investigated including geometric realization, categorification
and the connection with crystal bases. In \cite{EKL, EL, HW, KKO1,
KKO2, KKT, Wang}, various versions of Khovanov-Lauda-Rouquier {\it
super}-algebras have been introduced and the corresponding {\it
super}-categorifications have been constructed.

\end{remark}

\vskip 5mm

\section{Quantum affine algebras and $R$-matrices}

In this section, we briefly review the finite dimensional
representation theory of quantum affine algebras and the properties
of $R$-matrices (see, for example, \cite{AK, CP94, CP96, Kas02}).

Let $(A, P, \Pi, P^{\vee}, \Pi^{\vee})$ be a Cartan datum of affine
type with $I=\{0,1, \ldots, n\}$ the index set of simple roots. Let
$0\in I$ be the leftmost vertex in the affine Dynkin diagrams given
in \cite[Chapter 4]{Kac}. Set $I_{0} = I \setminus \{0\}$. Take
relatively prime positive integers $c_j$'s and $d_j$'s $(j \in I)$
such that
$$\sum_{j\in I} c_j a_{ji}=0, \quad \sum_{j \in I} a_{ij} d_j =0 \ \
\text{for all} \ i \in I.$$ Then the weight lattice can be written
as $$P=\bigoplus_{i \in I} \Z \Lambda_i \oplus \Z \delta,$$ where
$\delta:=\sum_{i\in I} d_i \alpha_i \in P$. We also define
$c:=\sum_{i\in I} c_i h_i \in P^{\vee}$.

\vskip 2mm

We denote by $\g$ the affine Kac-Moody algebra associated with $(A,
P, P^{\vee}, \Pi, \Pi^{\vee})$ and let $\g_{0}$ be the finite
dimensional simple Lie algebra inside $\g$ generated by $e_i$,
$f_i$, $h_i$ $(i \in I_{0})$. We will write $W$ and $W_{0}$ for the
Weyl group of $\g$ and $\g_{0}$, respectively.

\vskip 2mm

Let $U_q(\g)$ be the corresponding quantum group and let
$U_{q}'(\g)$ be the subalgebra of $U_q(\g)$ generated by $e_i$,
$f_i$, $K_i^{\pm 1}$ $(i \in I)$. The algebra $U_{q}'(\g)$ will be
called the {\it quantum affine algebra}.

Set $P_{\text{cl}}:= P \big/ \Z \delta$ and let $\text{cl}: P
\rightarrow P_{\text{cl}}$ be the canonical projection. Then we have
$$P_{\text{cl}} = \bigoplus_{i \in I} \Z \text{cl}(\Lambda_i) \ \
\text{and} \ \ P_{\text{cl}}^{\vee}:=\Hom_{\Z}(P_{\text{cl}},
\Z)=\bigoplus_{i \in I} \Z h_i.$$

A $U_{q}'(\g)$-module $V$ is {\it integrable} if

\ \ \ (i) $V=\bigoplus_{\lambda \in P_{\text{cl}}} V_{\lambda}$,
where $V_{\lambda}=\{v \in V \mid \, K_i v = q_i^{\langle h_i,
\lambda \rangle} v \ \text{for all} \ i \in I \}$,

\ \ \ (ii) $e_i$, $f_i$ $(i \in I)$ are locally nilpotent on $V$.

\noindent We denote by ${\mathcal C}_{\text{int}}$ the category of
finite dimensional integrable $U_{q}'(\g)$-modules.

\vskip 2mm

Let $M$ be an integrable $U_{q}'(\g)$-module. A weight vector $v \in
M_{\lambda}$ $(\lambda \in P_{\text{cl}})$ is called an {\it
extremal weight vector} if there exists a family of nonzero vectors
$\{v_{w \lambda} \mid \, w \in W \}$ such that
\begin{equation*}
v_{s_i \lambda} = \begin{cases} f_{i}^{(\langle h_i, \lambda
\rangle)}
v_{\lambda} \ \ & \text{if} \ \langle h_i, \lambda \rangle \ge 0, \\
e_{i}^{(-\langle h_i, \lambda \rangle)} v_{\lambda} \ \ & \text{if}
\ \langle h_i, \lambda \rangle \le 0.
\end{cases}
\end{equation*}

Let $P_{\text{cl}}^{0}:=\{\lambda \in P_{\text{cl}} \mid \, \langle
c, \lambda \rangle =0 \}$ and set
$$\varpi_{i}:= \Lambda_i - c_i \Lambda_0 \ \ \text{for} \
i \in I_{0}.$$ Then there exists a unique finite dimensional
integrable $U_{q}'(\g)$-module $V(\varpi_{i})$ satisfying the
following properties:

  (i) all the weights of $V(\varpi_{i})$ are contained in the
convex hull of $W_{0} \, \text{cl}(\varpi_{i})$.

  (ii) $\dim V(\varpi_{i})_{\text{cl}(\varpi_{i})} = 1$,

  (iii) for each $\mu \in W_{0} \, \text{cl}(\varpi_{i})$, there exists an
extremal weight vector of weight $\mu$,

  (iv) $V(\varpi_{i})$ is generated by
$V(\varpi_{i})_{\text{cl}(\varpi_{i})}$ as a $U_{q}'(\g)$-module.

\vskip 2mm

The $U_{q}'(\g)$-module $V(\varpi_{i})$ is called the {\it
fundamental representation of weight $\varpi_{i}$}\, $(i \in
I_{0})$.

Let $M$ be a $U_{q}'(\g)$-module. An involution on $M$ is called a
{\it bar involution} if $\overline{a\, v} = \overline{a} \
\overline{v}$ for all $a \in U_{q}'(\g)$, $v \in M,$ where
$\overline{e_i}=e_{i}$, $\overline{f_i}=f_i$,
$\overline{K_{i}}=K_{i}^{-1}$ $(i \in I)$. A finite
$U_{q}'(\g)$-crystal $B$ is {\it simple}\, if \,(i) $\text{wt}(B)
\subset P_{\text{cl}}^{0}$, \, (ii) there exists $\lambda \in
\text{wt}(B)$ such that $\#(B_{\lambda})=1$, \, (iii) the weight of
every extremal vector of $B$ is contained in $W_{0}\, \lambda$.

A finite dimensional integrable $U_{q}'(\g)$-module $M$ is {\it
good} \, if

\ \ \ (i) $M$ has a bar involution, \,

\ \ \ (ii) $M$ has a crystal basis with simple crystal, \,

\ \ \ (iii) $M$ has a lower global basis.

\vskip 2mm  \noindent For example, all the fundamental
representations $V(\varpi_{i})$ $(i \in I_{0})$ are good. Every good
module is irreducible. For any good module $M$, there exists an
extremal weight vector $v$ of weight $\lambda$ such that $\text{wt}
(U_{q}'(\g) v) \subset \lambda - \sum_{i \in I_{0}} \Z_{\ge 0}
\text{cl}(\alpha_i)$. Such $\lambda$ is called  a {\it dominant
extremal weight} and $v$ is called a {\it dominant extremal weight
vector}.

\vskip 2mm

Take ${\mathbf k}=\overline{\C(q)} \subset \bigcup_{M>0}
\C((q^{1/m}))$. Let $M_{\text{aff}} = {\mathbf k}[z, z^{-1}]
\otimes_{\mathbf k} M$ be the {\it affinization} of $M$. For $v \in
M$ and $k \in \Z$, the action of $U_{q}'(\g)$ on $M_{\text{aff}}$ is
given by
\begin{equation*}
\begin{aligned}
& e_{i} (z^k \otimes v)=\begin{cases} z^{k+1} \otimes e_{0} v \ \
& \text{if} \ i=0, \\
z^k \otimes e_i v \ \ & \text{if} \ i \neq 0,  \end{cases} \\
&  f_{i} (z^k \otimes v)=\begin{cases} z^{k-1} \otimes f_{0} v \ \
& \text{if} \ i=0, \\
z^k \otimes f_i v \ \ & \text{if} \ i \neq 0, \end{cases} \\
& K_{i}^{\pm 1} (z^k \otimes v) = q_{i}^{\pm \langle h_{i},
\text{wt}(v) \rangle} (z^k \otimes v) \ \ (i \in I).
\end{aligned}
\end{equation*}
We define a $U_{q}'(\g)$-module automorphism $z_{M}:M_{\text{aff}}
\rightarrow M_{\text{aff}}$ of weight $\delta$  by
$$z^{k} \otimes v
\mapsto z^{k+1} \otimes v \ \ (v\in M, \, k \in \Z).$$

\vskip 2mm

Let $M_1$, $M_2$ be good $U_{q}'(\g)$-modules and let $u_{1}$,
$u_{2}$ be dominant extremal weight vectors of $M_1$ and $M_2$,
respectively. Set $z_{1} = z_{M_1}$ and $z_{2} = z_{M_2}$. Then
there exists a unique $U_{q}'(\g)$-module homomorphism
\begin{equation*}
R_{M_1, M_2}^{\text{norm}}(z_1, z_2) : (M_{1})_{\text{aff}} \otimes
(M_{2})_{\text{aff}} \longrightarrow {\mathbf k}(z_1, z_2)
\otimes_{{\mathbf k}[z_{1}^{\pm 1}, z_{2}^{\pm 1}]}
(M_{2})_{\text{aff}} \otimes (M_{1})_{\text{aff}}
\end{equation*}
satisfying
\begin{equation*}
\begin{aligned}
&  R_{M_{1}, M_{2}}^{\text{norm}}(u_1 \otimes u_2) =u_2 \otimes
u_1,\\
& \  R_{M_{1}, M_{2}}^{\text{norm}} \circ z_1  = z_1 \circ
R_{M_{1},M_{2}}^{\text{norm}}, \\
& \  R_{M_{1}, M_{2}}^{\text{norm}} \circ z_2  = z_2 \circ R_{M_{1},
M_{2}}^{\text{norm}}.
\end{aligned}
\end{equation*}
The homomorphism  $R_{M_{1}, M_{2}}^{\text{norm}}$ is called the
{\it normalized $R$-matrix of $M_1$ and $M_2$}.

\vskip 2mm

Note that $\text{Im}\, R_{M_{1}, M_{2}}^{\text{norm}} \subset
{\mathbf k}(z_2 \big/ z_1) \otimes_{{\mathbf k}[(z_2 / z_1)^{\pm
1}]} (M_{2})_{\text{aff}} \otimes (M_{1})_{\text{aff}}.$ We denote
by $d_{M_1, M_2}(u) \in {\mathbf k}[u]$ the monic polynomial of the
smallest degree such that
$$\text{Im}\, \left( d_{M_1, M_2}(z_2 \big/ z_1)\, R_{M_{1}, M_{2}}^{\text{norm}} \right) \subset
(M_{2})_{\text{aff}} \otimes (M_{1})_{\text{aff}}.$$ The polynomial
$d_{M_1, M_2}(u)$ is called the {\it denominator} of $R_{M_1,
M_2}^{\text{norm}}$.

The normalized $R$-matrix satisfies the Yang-Baxter equation. That
is, for $U_{q}'(\g)$-modules $M_1$, $M_2$, $M_3$, we have
\begin{equation*}
\begin{aligned}
& (R_{M_2, M_3}^{\text{norm}} \otimes 1) \circ (1 \otimes  R_{M_1,
M_3}^{\text{norm}}) \circ ( R_{M_1, M_2}^{\text{norm}} \otimes 1) \\
&  \quad \quad  =(1 \otimes  R_{M_1, M_2}^{\text{norm}}) \circ
(R_{M_1, M_3}^{\text{norm}} \otimes 1) \circ(1 \otimes R_{M_2,
M_3}^{\text{norm}}).
\end{aligned}
\end{equation*}

\vskip 5mm

\section{Quantum affine Schur-Weyl duality functor}\label{sec:duality
functor}

Let $\{V_{s} \mid s \in {\mathcal S} \}$ be a family of good modules
and let $v_{s}$ be a dominant extremal weight vector in $V_{s}$ with
weight $\lambda_{s}$ $(s \in {\mathcal S})$. Take an index set $J$
endowed with the maps $X:J \rightarrow {\mathbf k}^{\times}$ and
$s:J \rightarrow {\mathcal S}$. For each $i,j \in J$, let
\begin{equation*}
\begin{aligned}
R_{V_{s(i)}, V_{s(j)}}^{\text{norm}} & (z_i, z_j)  :
(V_{s(i)})_{\text{aff}} \otimes (V_{s(j)})_{\text{aff}} \\
& \longrightarrow {\mathbf k}(z_i, z_j) \otimes_{{\mathbf
k}[z_{i}^{\pm 1}, z_{j}^{\pm 1}]} (V_{s(j)})_{\text{aff}} \otimes
(V_{s(i)})_{\text{aff}}
\end{aligned}
\end{equation*}
be the normalized $R$-matrix sending $v_{s(i)} \otimes v_{s(j)}$ to
$v_{s(j)} \otimes v_{s(i)}$.

\vskip 2mm

Let $d_{V_{s(i)}, V_{s(j)}} (z_j / z_i)$ be the denominator of
$R_{V_{s(i)}, V_{s(j)}}^{\text{norm}} (z_i, z_j)$. We define a
quiver $\Gamma^{J}$ as follows.

\ \ \ (i) We take $J$ to be the set of vertices.

\ \  (ii) We put $d_{ij}$ many arrows from $i$ to $j$, where
$d_{ij}$ the order of zero of $d_{V_{s(i)}, V_{s(j)}}(z_j / z_i)$ at
$z_{j} / z_{i} = X(j) / X(i)$.

\vskip 2mm

Define the Cartan matrix $A^{J}=(a_{ij}^{J})_{i,j \in J}$ by
\begin{equation} \label{eq:AJ}
a_{ij}^{J} = \begin{cases} 2 \ \ & \text{if} \ i=j, \\
-d_{ij} - d_{ji} \ \ & \text{if} \ i \neq j.
\end{cases}
\end{equation}
Thus we obtain a symmetric Cartan datum $(A^{J}, P, P^{\vee}, \Pi,
\Pi^{\vee})$ associated with $\Gamma^{J}$.

Set
\begin{equation}\label{eq:QJ}
{\sf Q}_{ij}^{J}(u,v):=\begin{cases} 0 \ \ & \text{if} \ i=j, \\
(u-v)^{d_{ij}} (v-u)^{d_{ji}} \ \ & \text{if} \ i \neq j.
\end{cases}
\end{equation}
We will denote by $R^{J}(\beta)$ $(\beta \in Q_{+})$ the
Khovanov-Lauda-Rouquier algebra associated with $(A^{J}, {\sf
Q}^{J})$.

For each $\nu=(\nu_1, \ldots, \nu_n) \in J^{\beta}$, let
$\widehat{{\mathcal O}}_{{\mathbf T}^n, X(\nu)}
 = {\mathbf k}[[X_1 - X(\nu_1), \ldots, X_n - X(\nu_n)]]$
be the completion of ${\mathcal O}_{{\mathbf T}^n, X(\nu)}$ at
$X(\nu):= (X(\nu_1), \ldots, X(\nu_n))$ and set
$$V_{\nu}:=(V_{s(\nu_1)})_{\text{aff}} \otimes \cdots \otimes
(V_{s(\nu_n)})_{\text{aff}},$$ where $X_{k}=z_{V_{s(\nu_k)}}$ $(k=1,
\ldots, n)$.

We define $$\widehat{V}_{\nu} := \widehat{{\mathcal O}}_{{\mathbf
T}^n, X(\nu)} \otimes_{{\mathbf k}[X_{1}^{\pm 1}, \ldots, X_{n}^{\pm
1}]} V_{\nu} \ \ \text{and} \ \ {\widehat{V}}^{\otimes \beta}:=
\bigoplus_{\nu \in J^{\beta}} \widehat{V}_{\nu} \, e(\nu).$$ The
following proposition is one of the main results of \cite{KKasKim1}.

\begin{proposition} {\rm \cite{KKasKim1}} {\rm The space
$\widehat{V}^{\otimes \beta}$ is a $(U_{q}'(\g),
R^{J}(\beta))$-bimodule.}
\end{proposition}

\vskip 2mm

Hence we obtain a functor $${\mathcal F}_{\beta}: \text{mod}\,
(R^{J}(\beta)) \longrightarrow \text{mod}\, U_{q}'(\g)$$ defined by
$$M \longmapsto \widehat{V}^{\otimes \beta} \otimes_{R^{J}(\beta)}
M  \qquad \text{for} \ \ M \in \text{mod} \, (R^{J}(\beta)).$$

Write $\text{mod}\, (R^{J}) : = \bigoplus_{\beta \in Q_{+}}
\text{mod}\, (R^{J}(\beta))$ and set
$${\mathcal F}=\bigoplus_{\beta \in Q_{+}} {\mathcal F}_{\beta}: \text{mod}\,(R^{J})
\longrightarrow \text{mod}\,U_{q}'(\g).$$ The functor $\mathcal{F}$
is called the {\it quantum affine Schur-Weyl duality functor}. The
basic properties of $\mathcal{F}$ are summarized in the following
theorem.

\begin{theorem}{\rm \cite{KKasKim1}} \label{thm:F} \hfill

{\rm

(a) The functor $\mathcal{F}$ restricts to
$$\mathcal{F} : \text{rep}\,(R^{J})
\longrightarrow {\mathcal C}_{\text{int}},$$ where
$\text{rep}\,(R^{J}):= \bigoplus_{\beta \in Q_{+}} \text{rep}\,
(R^{J}(\beta))$ and  ${\mathcal C}_{\text{int}}$ denotes the
category of finite dimensional integrable $U_{q}'(\g)$-modules.

(b) For each $i \in J$, let $S(\alpha_i):={\mathbf k}\,u(i)$ be the
$1$-dimensional graded simple $R^{J}(1)$-module defined by
$$e(j)\ u(i) = \delta_{ij}\, u(i), \ \ x_{1} \, u(i) = 0.$$
Then we have
$${\mathcal F}(S(\alpha_i)) \cong (V_{s(i)})_{X(i)},$$
where $(V_{s(i)})_{X(i)}$ is the evaluation module of $V_{s(i)}$ at
$z_{i}=X(i)$.

(c) ${\mathcal F}$ is a tensor functor; i.e., there exists a
canonical $U_{q}'(\g)$-module isomorphisms
$${\mathcal F}(R^{J}(0)) \cong {\mathbf k}, \quad {\mathcal F}(M
\circ N) \cong {\mathcal F}(M) \otimes {\mathcal F}(N)$$ for $M \in
\text{rep}\, (R^{J}(m))$, $N \in \text{rep}\,(R^{J}(n))$.

(d) If the quiver $\Gamma^{J}$ is of type $A_{n}$ $(n \ge 1)$,
$D_{n}$ $(n \ge 4)$, $E_{6}$, $E_{7}$, $E_{8}$, then ${\mathcal F}$
is exact. }
\end{theorem}

\vskip 5mm

\section{The Categories ${\mathcal T}_{N}$ and ${\mathcal C}_{N}$}
\label{sec:category T and C}

Take ${\mathbf k}= \C(q)$.  Let $\g = A_{N-1}^{(1)}$ be the affine
Kac-Moody algebra of type $A_{N-1}^{(1)}$ and let $V=V(\varpi_{1})$
be the fundamental representation of $U_{q}'(A_{N-1}^{(1)})$ of
weight $\varpi_{1}$.

Set ${\mathcal S}=\{V\}$, $J=\Z$ and let $X:\Z \rightarrow {\mathbf
k}^{\times}$ be the map given by $j \mapsto q^{2j}$ $(j \in \Z)$.
Then the normalized $R$-matrix $R_{V,V}^{\text{norm}}: V_{z_1}
\otimes V_{z_2} \longrightarrow V_{z_2} \otimes V_{z_1}$ has the
denominator $d_{V,V}(z_2 / z_1) = z_2 / z_1 - q^2.$ Hence we have
$$d_{ij}=\begin{cases} 1 \ \ & \text{if} \ j=i+1, \\
0 \ \ & \text{otehrwise},
\end{cases}$$
which yields the quiver $\Gamma^{J}$ of type $A_{\infty}$. Take
$P_{J}=\bigoplus_{k \in \Z} \Z \, \varepsilon_{k}$ to be the weight
lattice and let $Q^{J}=\bigoplus_{k \in \Z} \Z \,(\varepsilon_{k} -
\varepsilon_{k+1})$ be the root lattice. There is a bilinear form on
$P_{J}$ given by $(\varepsilon_a, \varepsilon_b)=\delta_{ab}$.

\vskip 2mm

For $a \le b$, let $l=b-a+1$ and let $L(a,b):= {\mathbf k}\, u(a,b)$
be the 1-dimensional graded simple $R^{J} (\varepsilon_a -
\varepsilon_{b+1})$-module defined by
\begin{equation*}
\begin{aligned}
& x_{s} \, u(a,b)=0, \quad \tau_{t}\, u(a,b)=0 \ \ (1 \le s \le l, \ 1 \le t \le l-1), \\
& e(\nu)\, u(a,b) = \begin{cases} u(a,b) \ \ & \text{if} \ \nu=(a,
a+1, \ldots, b), \\ 0 \ \ & \text{otherwise}.
\end{cases}
\end{aligned}
\end{equation*}
Then we have
\begin{equation*}
{\mathcal F}(L(a,b)) \cong \begin{cases} V(\varpi_{l})_{(-q)^{a+b}}
\ \ & \text{if} \ 0 \le l \le N , \\
0 \ \ & \text{if} \ l>N,
\end{cases}
\end{equation*}
where ${\mathcal F}:\text{mod}\, (R^{J}(l)) \longrightarrow
\text{mod} \, U_{q}'(\g)$ is the  quantum affine Schur-Weyl duality
functor.

\vskip 2mm

Recall that $\text{Rep}\, (R^J(l))$ is the category of finite
dimensional graded $R^J(l)$-modules. Set ${\mathcal
R}:=\bigoplus_{l\ge 0} \text{Rep} \, (R^J(l))$ and let ${\mathcal
S}$ be the Smallest Serre subcategory of ${\mathcal R}$ such that

\vskip 2mm

\ \ \ (i) ${\mathcal S}$ contains $L(a, a+N)$ for all $a \in \Z$,

\ \ \ (ii) $X \circ Y, Y \circ X \in {\mathcal S}$ for all $X \in
{\mathcal R}$, $Y \in {\mathcal S}$.

Take the quotient category ${\mathcal R} \big/ {\mathcal S}$ and let
${\mathcal Q}: {\mathcal R} \rightarrow {\mathcal R} \big/ {\mathcal
S}$ be the canonical projection functor. Then we have:

\begin{proposition}{\rm \cite{KKasKim1}} \hfill

{\rm (a) The functor ${\mathcal F}$ factors through ${\mathcal R}
\big/ {\mathcal S}$. That is, there is a canonical functor
${\mathcal F}_{\mathcal S}: {\mathcal R} \big/ {\mathcal S}
\longrightarrow \text{mod}\, U_{q}'(\g)$ such that the following
diagram is commutative.

\begin{equation*}
\xymatrix{  {\mathcal R}  \ar[dd]_{\mathcal Q} \ar[rrrr]^{\mathcal
F}   & & & & \text{mod}\,U_{q}'(\g) \\
            {}                           & &  & {} \\
            {\mathcal R} \big/ {\mathcal S} \ar[uurrrr]_{{\mathcal F}_{\mathcal S}}             & & & & {} }
\end{equation*}

(b) The functor ${\mathcal F}_{\mathcal S}$ sends a simple object in
${\mathcal R}\big/ {\mathcal S}$ to a simple object in $\text{mod}\,
U_{q}'(\g)$. }
\end{proposition}

\vskip 2mm

Let $L_{a}:=L(a, a+N-1)$ and $u_{a}:=u(a, a+N-1) \in L_{a}$ $(a \in
\Z)$. Then ${\mathcal F}(L_{a})$ is isomorphic to the trivial
representation of $U_{q}'(\g)$. Let $S:P_{J} \rightarrow P_{J}$
$(\varepsilon_{a} \mapsto \varepsilon_{a+N-1})$ be an automorphism
on $P_{J}$ and let $B$ be the bilinear form on $P_{J}$ given by
$$B(x,y):= - \sum_{k>0} (S^k x, y) \ \ \text{for all} \ x, y \in
P_{J}.$$ We define a new tensor product $\star$ on ${\mathcal
R}\big/ {\mathcal S}$ by
$$X \star Y:= q^{B(\alpha, \beta)}\, X \circ Y \quad \text{for} \ \ X
\in ({\mathcal R}\big/ {\mathcal S})_{\alpha}, \ Y \in ({\mathcal
R}\big/ {\mathcal S})_{\beta}.$$ Then there exists an isomorphism
$R(a)(X): L_{a} \star X \overset{\sim} \longrightarrow X \star
L_{a}$ which is functorial in $X \in {\mathcal R}\big/ {\mathcal
S}$. Moreover, the isomorphisms
$$R_{a}(L_{b}): L_{a} \star L_{b} \overset {\sim} \longrightarrow
L_{b} \star L_{a} \ \ \text{and} \ \ R_{b}(L_{a}): L_{b} \star L_{a}
\overset {\sim} \longrightarrow L_{a} \star L_{b}$$ are inverses to
each other. One can verify that $\{L_{a}, R_{a}(L_{b}) \mid \, a, b
\in \Z \}$ forms a commuting family of central objects in
$({\mathcal R} \big/ {\mathcal S}, \star)$ (see \cite[Appendix
A.6]{KKasKim1}).

\vskip 2mm

Let ${\mathcal T}_{N}':= ({\mathcal R} \big/ {\mathcal S}) [
L_{a}^{\star -1} \mid \, a\in \Z]$ \, be the localization of
${\mathcal R} \big/ {\mathcal S}$  by this commuting family and
define
$${\mathcal T}_{N}:= ({\mathcal R} \big/ {\mathcal S}) [
L_{a}\cong {\mathbf 1} \mid \, a\in \Z].$$ We denote by ${\mathcal
P}:{\mathcal R}\big/ {\mathcal S} \rightarrow {\mathcal T}_{N}$ the
canonical functor.

\begin{theorem}{\rm \cite{KKasKim1}} \hfill

{\rm (a) The category ${\mathcal T}_{N}$ is a rigid tensor category
; i.e., every object in ${\mathcal T}_{N}$ has a right dual and a
left dual.

(b) The functor  ${\mathcal F}_{\mathcal S}$ factors through
${\mathcal T}_{N}$. That is, there exists a canonical functor
${\mathcal F}_{N}: {\mathcal T}_{N} \longrightarrow \text{mod}\,
U_{q}'(\g)$ such that the following diagram is commutative.
\begin{equation*}
\xymatrix{ {\mathcal R} \ar[dd]_{\mathcal Q} \ar[rrrr]^{\mathcal F} & & &  & \text{mod}\, U_{q}'(\g) \\
{} & & & & {} \\
{\mathcal R} \big/ {\mathcal S} \ar[rrrr]_{\mathcal P}
\ar[uurrrr]^{{\mathcal F}_{\mathcal S}} & & & & {\mathcal T}_{N}
\ar[uu]_{{\mathcal F}_{N}} }
\end{equation*}

(c) The functor ${\mathcal F}_{N}$ is exact and sends a simple
object in ${\mathcal T}_{N}$ to a simple object in $\text{mod}\,
U_{q}'(\g)$. }

\end{theorem}

\vskip 2mm

Let ${\mathcal C}_{N}$ be the smallest full subcategory of
${\mathcal C}_{\text{int}}$ consisting of $U_{q}'(\g)$-modules $M$
such that every composition factor of $M$ appears as a composition
factor of a tensor product of modules of the form
$V(\varpi_{1})_{q^{2j}}$ $(j \in \Z)$. Thus ${\mathcal C}_{N}$ is an
abelian category containing all $U_{q}'(\g)$-modules
$V(\varpi_{i})_{(-q)^{i+2a-1}}$ for $1 \le i \le N-1$ and $a \in
\Z$. Moreover, ${\mathcal C}_{N}$ is stable under taking submodules,
quotients, extensions and tensor products. Hence ${\mathcal F}_{N}$
restricts to an exact functor
$${\mathcal F}_{N} : {\mathcal T}_{N} \longrightarrow
{\mathcal C}_{N}.$$

Let $Irr\,({\mathcal T}_{N})$ (respectively, $Irr\, ({\mathcal
C}_{N})$) denote the set of isomorphism classes of simple objects in
${\mathcal T}_{N}$ (respectively, in ${\mathcal C}_{N}$). Define an
equivalence relation on $Irr\, ({\mathcal T}_{N})$ by setting \ $X
\sim Y$ if and only if $X \overset{\sim} \rightarrow q^m Y$ for some
$m \in \Z$. Set $$Irr\, ({\mathcal T}_{N})|_{q=1}:= Irr\,({\mathcal
T}_{N}) \big/ \sim.$$

\begin{theorem} {\rm \cite{KKasKim1}} \hfill

{\rm (a) The functor ${\mathcal F}_{N}$ induces a bijection between
$Irr\, ({\mathcal T}_{N})|_{q=1}$ and $Irr\,({\mathcal C}_{N})$.

(b) The exact functor ${\mathcal F}_{N}$ induces a ring isomorphism
$$ \phi_{N} : K({\mathcal T}_{N})|_{q=1}
\overset {\sim} \longrightarrow K({\mathcal C}_{N}).$$

Therefore, the category ${\mathcal T}_{N}$ provides a {\it graded
lifting} of ${\mathcal C}_{N}$ as a rigid tensor category. }
\end{theorem}

\vskip 5mm

\section{The category ${\mathcal C}_{Q}$} \label{sec:category Q}

In this section,  we deal with affine Kac-Moody algebras $\g$ of
type $A_{n}^{(1)}$ $(n \ge 1)$, $D_{n}^{(1)}$ $(n \ge 4)$,
$E_{6}^{(1)}$, $E_{7}^{(1)}$, $E_{8}^{(1)}$. Let $I=\{0, 1, \ldots,
n\}$ be the index set for the simple roots of $\g$ and set $I_{0} =
I \setminus \{0\}$. We denote by $\g_{0}$ the finite dimensional
simple Lie subalgebra of $\g$ generated by $e_i, f_i, h_i$ $(i \in
I_{0})$. Thus $\g_{0}$ is of type $A_{n}$ $(n\ge 1)$, $D_{n}$ $(n
\ge 4)$, $E_{6}$, $E_{7}$, $E_{8}$, respectively.

Let $Q$ be the Dynkin quiver associated with $\g_{0}$. A function
$\xi: I_{0} \rightarrow \Z$ is called a {\it height function} if
$\xi_{j}=\xi_{i}-1$ whenever we have an arrow $i \rightarrow j$.

Set $$\widehat{I_0}:=\{(i, p) \in I_{0} \times \Z \mid \, p -
\xi_{i} \in 2 \Z \}.$$ The {\it repetition quiver} $\widehat{Q}$ is
defined as follows.

\ \ \ (i) We take $\widehat{I_{0}}$ to be the set of vertices.

\ \ \ (ii) The arrows are given by
$$(i,p) \rightarrow (j, p+1), \quad (j,q) \rightarrow (i, q+1)$$
for all arrows $i \rightarrow j$ and $p, q \in Z$ such that $p -
\xi_{i} \in \Z$, $q - \xi_{j} \in \Z$.

\vskip 2mm

For all $i \in I_{0}$, let $s_{i}(Q)$ be the quiver obtained from
$Q$ by reversing the arrows that touch $i$. A reduced expression
$w=s_{i_1} \cdots s_{i_l} \in W_{0}$ is said to be {\it adapted to
$Q$} if $i_k$ is the source of $s_{i_{k-1}}\cdots s_{i_1}(Q)$ for
all $1 \le k \le l$. It is known that there is a unique Coxeter
element $\tau\in W_{0}$ which is adaped to $Q$.

Set $\widehat{\Delta}:= \Delta_{+} \times \Z$, where $\Delta_{+}$ is
the set of positive roots of $\g_{0}$. For each $i \in I_{0}$, let
$B(i) := \{ j\in I_{0}  \mid \, \text{there is a path from $j$ to
$i$} \}$ and define $\gamma_{i}:=\sum_{j\in B(i)} \alpha_j$. We
define a bijection $\phi: \widehat{I_0} \rightarrow {\widehat
\Delta}$ inductively as follows.

\ \ (1) We begin with $\phi(i, \xi_i):=(\gamma_i, 0)$.

\ \ (2) If $\phi(i,p) = (\beta, j)$ is given, then we define

\ \ \ \ \ $\bullet$ $\phi(i, p-2):= (\tau(\beta), j)$ \ if
$\tau(\beta) \in \Delta_{+}$,

\ \ \ \ \ $\bullet$ $\phi(i, p-2):= (-\tau(\beta), j-1)$ \ if
$\tau(\beta) \in \Delta_{-}$,

\ \ \ \ \ $\bullet$ $\phi(i, p+2):= (\tau^{-1}(\beta), j)$ \ if
$\tau^{-1}(\beta) \in \Delta_{+}$,

\ \ \ \ \ $\bullet$ $\phi(i, p+2):= (-\tau^{-1}(\beta), j+1)$ \ if
$\tau^{-1}(\beta) \in \Delta_{-}$.

Let $w_{0}$ be the longest element of $W_{0}$ and fix a reduced
expression $w_{0}=s_{i_1} \cdots s_{i_l}$ which is adapted to $Q$.
Set
$$J:= \{(i,p) \in \widehat{I_{0}} \mid \, \phi(i,p) \in \Pi_{0} \times
\{0\}\},$$ where $\Pi_{0}$ denotes the set of simple roots of
$\g_{0}$. Take the maps $X: J \rightarrow {\mathbf k}^{\times}$ and
$s:J \rightarrow \{V(\varpi_{i}) \mid \, i \in I_{0} \}$ defined by
\begin{equation*}
X(i,p) = (-q)^{p+h}, \quad   s(i,p) = V(\varpi_{i}) \ \ \ \text{for}
\ (i,p) \in J,
\end{equation*}
where $h$ is the Coxeter number of $\g_{0}$.

\begin{theorem}{\rm \cite{KKasKim2}} \label{thm:g0} \hfill

{\rm For any $(i,p), (j,r) \in J$, assume that the normalized
$R$-matrix $R_{V(\varpi_{i}), V(\varpi_{j})}^{\text{norm}} (z)$ has
a pole at $z=(-q)^{r-p}$ of order at most $1$. Then the following
statements hold.

\vskip 2mm

(a) The Cartan matrix $A^{J}$ associated with $(J, X, s)$ is of type
$\g_0$.

(b) There exists a quiver isomorphism
$$Q^{\text{rev}} \overset{\sim} \longrightarrow \Gamma^{J}, \quad k \mapsto
\phi^{-1} (\alpha_k, 0) \ \ (k \in I_{0}), $$ where $Q^{\text{rev}}$
is the reverse quiver of $Q$.

(c) The functor ${\mathcal F} : {\rm rep} \,( R^{J}) \rightarrow
{\mathcal C}_{\text{int}}$ is exact and
$${\mathcal F}(S(\alpha_k)) \cong V(\varpi_{i})_{(-q)^{p+h}},$$
where $\phi(i,p)=(\alpha_k, 0)$. }
\end{theorem}

\begin{remark}
When $\g$ is of type $A_{n}^{(1)}\ (n \ge 1)$ or $D_{n}^{(1)} \
(n\ge 4)$, then the condition in Theorem \ref{thm:g0} is satisfied.
We conjecture that the same is true of $\g = E_{6}^{(1)}, \,
E_{7}^{(1)}, \, E_{8}^{(1)}$.
\end{remark}

\vskip 2mm

We now bring out the main subject of our interest in this section,
the category ${\mathcal C}_{Q}$ (cf. \cite{HL11}). Let ${\mathcal
C}_{Q}$ be the smallest abelian full subcategory of ${\mathcal
C}_{\text{int}}$ such that

\ \ \ (i) ${\mathcal C}_{Q}$ is stable under taking submodules,
subquotients, direct sums and tensor products,

\ \ \ (ii) ${\mathcal C}_{Q}$ contains all $U_{q}'(\g)$-modules of
the form $V(\beta)_{z} \big/ (z-1)^{l} \, V(\beta)_{z}$ \, $(\beta
\in \Delta_{+}, \, l \ge 1)$. Here,
$V(\beta)=V(\varpi_{i})_{(-q)^{p+h}}$ such that $\phi(i,p) = (\beta,
0)$.

Let $\text{Nilrep}\,( R^{J}(\beta))$ be the category of finite
dimensional ungraded $R^{J}(\beta)$-modules such that all $x_k$'s
act nilpotently and set
$$\text{Nilrep}\,( R^{J}):= \bigoplus_{\beta \in Q_{+}}
\text{Nilrep}\,( R^{J}(\beta)).$$ Note that every module in
$\text{Nilrep}\,(R^{J})$ can be obtained by taking submodules,
subquotients, direct sums and convolution products of $P(\alpha_k)
\big/\,(x_{1}^l)$ $(k \in I_{0}, l \ge 0)$, where $P(\alpha_k)$ is
the projective cover of $S(\alpha_k)$. Thus we obtain a well-defined
functor $${\mathcal F}: \text{Nilrep} \, (R^{J}) \longrightarrow
{\mathcal C}_{Q},$$ which satisfies the following properties.

\begin{theorem} {\rm \cite{KKasKim1, KKasKim2}} \label{thm:CQ}
\hfill

{\rm (a) ${\mathcal F}$ is an exact tensor functor.

(b)  ${\mathcal F}$ sends a simple object in $\text{Nilrep}
\,(R^{J})$ to a simple object in ${\mathcal C}_{Q}$. }
\end{theorem}

It is straightforward to verify that ${\mathcal F}$ is a faithful
functor. Since ${\mathcal C}_{Q}$ is the smallest abelian full
subcategory of ${\mathcal C}_{\text{int}}$ satisfying the conditions
(i) and (ii) given above, we conjecture that ${\mathcal F}$ is full
and defines an equivalence of categories.

\begin{remark}
Note that our general approach to quantum affine Schur-Weyl duality
applies to {\it all} quantum affine algebras and {\it any} choice of
good modules. Thus we expect there are a lot more exciting
developments to come.  It is an interesting question whether our
general construction can shed a new light on the hidden connection
between quantum affine algebras and cluster algebras (cf.
\cite{HL10}).

\end{remark}

\end{document}